\documentclass{amsart}

\usepackage{amsmath, amssymb, graphicx, epsfig, verbatim}
\usepackage{epic}

\newtheorem{thm}{Theorem}[section]

\newtheorem{cor}[thm]{Corollary}
\newtheorem{lem}[thm]{Lemma}
\newtheorem{prop}[thm]{Proposition}

\newtheorem{defn}[thm]{Definition}
\newtheorem{conj}[thm]{Conjecture}

\newtheorem{rem}[thm]{Remark}

\numberwithin{equation}{section}
\newcommand{\defin}[1]{{\bf\emph{#1}}}
\newcommand{\Field}{{\mathbb {F}}}
\newcommand{\HFKa}{{\widehat {\rm {HFK}}}}

\newcommand{\Z}{\mathbb Z}
\newcommand{\N}{\mathbb N}
\newcommand{\Q}{\mathbb Q}
\newcommand{\R}{\mathbb R}

\begin{document}

\title{A note on thickness of knots}

\author{Andr\'{a}s I. Stipsicz}
\address{R\'enyi Institute of Mathematics\\
H-1053 Budapest\\ 
Re\'altanoda utca 13--15, Hungary}
\email{stipsicz.andras@renyi.hu}

\author{Zolt\'an Szab\'o}
\address{Department of Mathematics\\
Princeton University\\
 Princeton, NJ, 08544}
\email{szabo@math.princeton.edu}

\begin{abstract}
  We introduce a numerical invariant $\beta(K)\in \N$ of a knot
  $K\subset S^3$ which measures how non-alternating $K$ is.  We prove
  an inequality between $\beta (K)$ and the (knot Floer) thickness
  $th(K)$ of a knot $K$. As an application we show that all Montesinos
  knots have thickness at most one.
\end{abstract}
\maketitle

\section{Introduction}
\label{sec:intro}

A knot $K\subset S^3$ is \emph{alternating} if it admits a diagram 
with the property that when traversing through the diagram, we alternate
between over- and under-crossings. (An intrinsic definition of alternating
knots have been recently found by Greene and Howie~\cite{Josh, Howie}.)
A diagram of $K$ partitions the plane into domains (the connected components
of the complemet of the projection), and the alternating property
can be rephrased by saying that on the boundary of each domain each edge
connects an under-crossing with an over-crossing. Indeed, this observation
provides a way to measure how far a knot is from being alternating.
We introduce the following definition:

\begin{defn}
  Suppose that $D$ is the diagram of a given knot $K\subset S^3$.
  A domain $d$ of $D$ is \defin{good} if any edge
  on the boundary of $d$ 
  connects an over- and an undercrossing. The domain $d$ is
  \defin{bad} if it is not good. The number of bad domains of the diagram
  $D$ is denoted by $B(D)$.
\end{defn}

Obviously the diagram $D$ is alternating iff $B(D)=0$. Indeed,
by taking
\[
\beta (K)=\min \{ B(D)\mid D\ \mbox{diagram\ for\ }\ K\}
\]
we get a knot invariant, which satisfies that $\beta (K)=0$ if and
only if $K$ is an alternating knot. As it is typical for knot
invariants given by minima of quantites over all diagrams, it is easy
to find an upper bound on $\beta (K)$ (by determining $B(D)$ for a
diagram of $K$), but it is harder to actually compute its value.

As it turns out, knot Floer homology provides a lower bound for $\beta
(K)$ through the \emph{thickness} of $K$. Recall that $\HFKa (K)$, the
hat-version of knot Floer homology of $K$ is a finite dimensional
bigraded vector space over the field $\Field$ of two elements. By
collapsing the Maslov and Alexander gradings $M$ and $A$ on $\HFKa
(K)$ to $\delta =A-M$ we get a graded vector space $\HFKa
^{\delta}(K)$. The thickness $th(K)$ of $K$ is the largest possible
difference of $\delta$-gradings of two homogeneous (nonzero) elements
of this vector space. It is known that for an alternating knot $K$ the
$\delta$-graded Floer homology is in a single $\delta$-grading
(determined by the signature of the knot), hence if $K$ is
alternating, then $th(K)=0$. (Knots satisfying $th(K)=0$ are called
\emph{thin} knots, hence alternating knots are thin.)

With this definition in place, the main result of this paper is
\begin{thm}\label{thm:estimate}
Suppose that $K\subset S^3$ is a non-alternating knot. Then
    \begin{equation}\label{eq:ThicknessBadness}
    th(K)\leq \frac{1}{2}\beta (K)-1.
    \end{equation}
\end{thm}

While the thickness of $K$ can be used to estimate how non-alternating
$K$ is, the formula of Equation~\eqref{eq:ThicknessBadness} also can
be used to estimate $th(K)$ by finding appropriate diagrams of $K$.
In particular, the formula can be applied to show

\begin{cor}(Lowrance~\cite{Lowrance})
  \label{cor:Monte}
  Suppose that $K$ is a Montesinos knot. Then, $th(K)\leq 1$.
\end{cor}

\begin{rem}
  A quantity similar to $\beta (K)$ has been introduced by Turaev~\cite{Turaev},
  now called the \defin{Turaev genus} $g_T(K)$. An inequality similar to
  Inequality~\eqref{eq:ThicknessBadness} for the Turaev genus and the
  (knot Floer) thickness $th(K)$ was shown by Lowrence in~\cite{Lowrance}.
  As the Turaev genus of non-alternating Montesinos knots is known
  to be equal to 1~\cite{ChampKof, FiveAuthors},
  our Corollary~\ref{cor:Monte} also follows from \cite{Lowrance}. 
\end{rem}

The formula of Inequality~\eqref{eq:ThicknessBadness} can be used in a
further way: by a recent result of Zibrowius~\cite{Zib} mutation does not
change $\HFKa ^{\delta}(K)$, hence leaves $th(K)$ unchanged. Consequently,
besides isotopies we can change a diagram by mutations to get better estimates
for $th(K)$ through $B(D)$ for a diagram $D$ of a mutant.

The paper is organized as follows. In Section~\ref{sec:KnotFloer} we
recall basics of knot Floer homology and prove the theorem stated
above. In Section~\ref{sec:Montesinos} we give the details of the
proof of Corollary~\ref{cor:Monte}, and finally in
Section~\ref{sec:further} we list some further properties and
questions regarding $\beta$.

{\bf {Acknowledgements}}: AS was partially supported by the
\emph{\'Elvonal (Frontier) project} of the NKFIH (KKP126683).  ZSz was
partially supported by NSF Grants DMS-1606571 and DMS-1904628. We
would like to thank Jen Hom and Tye Lidman for a motivating
discussion.

\section{The knot Floer homology thickness of knots}
\label{sec:KnotFloer}

Suppose that $V=\sum _aV_a$ is a finite dimensional graded vector space,
where $V_a\subset V$ is the subspace of homogeneous elements of
grading $a\in \R$. The \emph{thickness} $th(V)$ of $V$ is by definition the
largest possible difference between gradings of (non-zero) homogeneous elements:
\[
th(V)=\max \{ a\in \R \mid V_a\neq 0\}-\min \{a \in \R  \mid V_a\neq 0\}.
\]

Suppose now that the graded vector space $V$ is endowed with a
boundary operator $\partial$ of degree $1$; then the homology
$H(V,\partial )$ also admits a natural grading from the grading of $V$.
As $H(V, \partial )$ is the quotient of a subspace of $V$,
it is easy to see that
\[
th(H(V,\partial ))\leq th (V).
\]

The hat version of knot Floer homology (over the field $\Field$ of two
elements) of a knot $K\subset S^3$ is a finite dimensional bigraded
vector space $\HFKa (K)=\sum _{M,A} \HFKa _M (K, A)$. By collapsing
the two gradings to $\delta =A-M$, we get the $\delta$-graded
invariant ${\HFKa }^{\delta}(K)$. The thickness of ${\HFKa
}^{\delta}(K)$ is by definition the thickness $th(K)$ of $K$.

Knot Floer homology is defined as the homology of a chain complex
we can associate to a diagram of the knot (and some further choices).
Indeed, for a given  diagram $D$ of a knot $K$ fix a marking, that is a
point of $D$ which is not a crossing. 
Consider the bigraded vector space $C_{D,p}$ (graded by the Alexander
and the Maslov gradings $A$ and $M$) associated to the marked
diagram $(D,p)$, which is  generated over $\Field$ by the
Kauffman states of the marked diagram, a concept which we recall below.

Suppose that for the marked diagram $(D, p)$ of the knot $K$ 
the set of crossings is denoted by $Cr(D)$, the set of domains by
$Dom(D)$, and $Dom_p(D)$ denotes the set of those domains which do not contain
$p$ on their boundary. A \emph{Kauffman state} $\kappa$ is a bijection
$\kappa \colon Cr (D)\to Dom_p(D)$ with the property that for a
crossing $c\in Cr (D)$ the value $\kappa (c)$ is one of the (at most
four) domains meeting at $c$. The Alexander, Maslov and
$\delta$-gradings of a Kauffman state are computed by summing the
local contributions at each crossing, as given by the diagrams of
Figure~\ref{fig:Contributions}.
\begin{figure}
\centering
\includegraphics{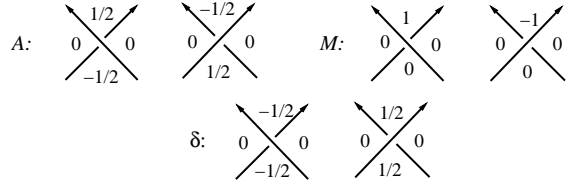}
\caption{{\bf The local contributions for $A, M$ and $\delta$ at a
    crossing.}  The Kauffman state distinguishes a corner at the
  crossing, and we take the value in that corner as a contribution of
  the crossing to $A,M$ or $\delta$ of the Kauffman state at hand.}
\label{fig:Contributions} 
\end{figure}

According to \cite{OSzAlternating} there is a boundary map $\partial
\colon C_{D,p}\to C_{D,p}$
of bidegree $(-1,0)$ (in the bigrading $(M,A)$)
with the property that $H(C_{D,p}, \partial
)=\HFKa (K)$ is isomorphic to the knot Floer homology of $K$
(as a bigraded vector space). By collapsing the two
gradings $A$ and $M$ to $\delta =A-M$, we get the graded vector spaces
$(C_{D,p}^{\delta},\partial )$ and its homology ${\HFKa
}^{\delta}(K)$.  As $\HFKa^{\delta }(K)$ is the quotient of a subspace
of $C^{\delta}_{D,P}$, we have that $th({\HFKa }^{\delta}(K)) \leq
th(C_{D,p}^{\delta}, \partial)$.

\begin{prop}
  \label{prop:ThicknessBadDomains}
  Suppose that $D$ is a diagram of the knot $K$.  If $D$ is not an
  alternating diagram, then
  \[
 th(C^{\delta}_{D,p})\leq \frac{1}{2}B(D)-1.
\]  
  \end{prop}
\begin{proof}
Fix a marked point $p$ on $D$, and 
  consider the $\delta$-graded chain complex
  $(C^{\delta}_{D,p},\partial )$ generated by the Kauffman states of $(D,p)$.

  The $\delta$-grading at a positive crossing is either 0 or $-\frac{1}{2}$,
and at a negative crossing it is either 0 or $\frac{1}{2}$.
So we can express the $\delta$-grading of a
  Kauffman state $\kappa$ as the sum
  \[
  -\frac{1}{4}{\rm {wr}}(D)+\sum _{c\in Cr}f(\kappa (c)),
  \]
  where ${\rm {wr}}$ is the writhe of the diagram, and $f$ is a function
  on the Kauffman corners, which is either $\frac{1}{4}$ or $-\frac{1}{4}$
  (depending on the chosen quadrant at the crossing $c$).

  Simple computation shows that for a good domain each corner in the
  domain gives the same $f$-value, hence for different Kauffman states
  the contributions from this particular domain are the same. This is
  no longer true for a bad domain, but the difference of two
  contributions is at most $\frac{1}{2}$.  When determining the possible
  maximum of $\delta (x)-\delta (x')$ for two homogeneous elements
  $x,x'\in C_{D,p}^{\delta}$, the contributions from the writhe cancel,
  and so do the contributions from good domains, while bad domains
  contribute at most $\frac{1}{2}$. This shows that
  $th(C_{D,p}^{\delta})\leq \frac{1}{2}B(D)$.

  By assumption $D$ is not alternating, hence there is a bad domain,
  with an edge showing that it is bad. Choose the marking $p$ on such an
  edge.  Since this edge guarantees that the two domains having it
 on their boundary are both bad, while these two bad domains do not
  get Kauffman corners, we get that $th(C_{D,p})$ is bounded by
  $\frac{1}{2}(B(D)-2)=\frac{1}{2}B(D)-1$, concluding the proof.
\end{proof}

\begin{proof}[Proof of Theorem~\ref{thm:estimate}]
  Suppose that $K$ is not alternating. Then any diagram $D$ of
  $K$ is non-alternating, hence we have that 
\[
  th(K)\leq th(C_{D,p}^{\delta})\leq \frac{1}{2}B(D)-1.
  \]
  Since $\beta (K)$ is computed from the minimum of the right-hand
  side of this inequality, the proof follows at once.
\end{proof}

\section{Montesinos knots}
\label{sec:Montesinos}

Montesinos knots are straighforward generalizations of pretzel knots;
a diagram involving rational tangles defining the Montesinos
knot $M(r_1, \ldots , r_n)$ is shown by Figure~\ref{fig:monte}. 
(A box with a rational number $r_i$ in it symbolizes the
tangle shown by Figure~\ref{fig:tangle}.)
\begin{figure}
\centering
\includegraphics{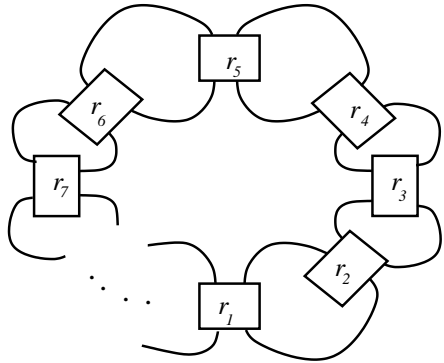}
\caption{{\bf The Montesinos knot $M(r_1, \ldots , r_n)$.}  The box
  containing $r_i$ denotes the algebraic tangle determined by the
  rational number $r_i=\frac{\beta _i}{\alpha _i}$
  (cf. Figure~\ref{fig:tangle}).  In order to have a knot, at most one
  of $\alpha _i$ can be even.}
\label{fig:monte} 
\end{figure}
\begin{figure}
\centering
\includegraphics[width=5cm]{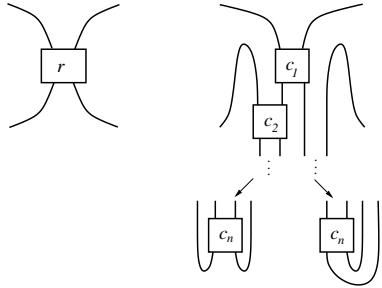}
\caption{{\bf The rational tangle corresponding to $r\in \Q$.}  Here
  the boxes with $c_i\in \Z$ on the right denote $\vert c_i\vert $
  half twists (right handed for positive, left handed for negative
  $c_i$). Depending on the parity of $n$ (the numbre of $c_i$) we have
  two different finishing forms. The rational number $r$ determines
  the coefficients $c_i$ through its continued fraction expansion. The
  tangle is alternating (as part of a knot or link) if $c_i$ alternate
  in signs.}
\label{fig:tangle} 
\end{figure}
We allow any of the $r_i$ to be equal to $\pm 1$. Notice that the
order of $(r_1, \ldots , r_n)$ is important; those $r_i$ which are equal
to $\pm 1$ can be commuted with any other parameter through a simple
isotopy of the diagram.

\begin{lem}\label{lem:isotopy}
  Consider the diagram of the Montesinos knot $M(r_1, \ldots , r_n)$ given
  by Figure~\ref{fig:monte}.
  It can be isotoped to a diagram with at most four bad domains.
\end{lem}

\begin{proof}
  Recall that a rational tangle has the form given by Figure~\ref{fig:tangle}.
Adapting the isotopies described in \cite{Goldman-Kauffman}, we can
achieve that all tangles are alternating, hence the potentially bad
domains are the ones between the tangles, together with the central
and the unbounded domains.  The number of bad domains
between the tangles can be reduced by the following observation.  The
domain between two tangles is bad if the first coefficients $c_1^1$
and $c_1^2$ of the two rational numbers determining the tangles have
opposite signs, say $c_1^1>0$ and $c_1^2<0$. Then by Reidemestier-II
moves we can introduce cancelling twistings, as shown by
Figure~\ref{fig:cancel}, and then commute the first twisting (in the figure
given by the box with $2$ in it) between the first and second tangles
of the Montesinos knot.
\begin{figure}
\centering
\includegraphics[width=5cm]{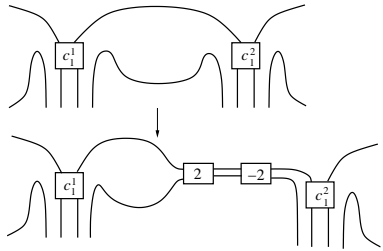}
\caption{\bf The introduction of cancelling twists to turn domains
    between tangles good.}
\label{fig:cancel} 
\end{figure}
All domains between the boxes will become good, except the ones
connecting the first tangle with the newly introduced twists and the
second tangle also connecting it with the newly introduced twists.
After these alterations make sure that (by the adaptation of
\cite{Goldman-Kauffman}) all tangles are isotoped to be alternating.
In total the new diagram then has four bad domains, concluding the
proof.
\end{proof}

\begin{proof}[Proof of Corollary~\ref{cor:Monte}]
  For a Montesinos knot $M(r_1, \ldots , r_n)$ an approrpiate isotopy
  of the diagram of Figure~\ref{fig:monte} (as given by
  Lemma~\ref{lem:isotopy}) gives a diagram with at most four bad
  domains. The application of Theorem~\ref{thm:estimate} concludes the
  argument.
\end{proof}

\begin{rem}
  Using the mutation invariance of $th(K)$, Lemma~\ref{lem:isotopy}
  can be avoided: by mutations any Montesinos knot $M(r_1, \ldots ,
  r_n)$ can be moved to $M(q_1, \ldots , q_n)$ with the same rational
  parameters in a different order so that $q_i$ and $q_{i+1}$ have the
  same sign with at most one exception. Isotoping the diagram so that
  the tangles are alternating, the mutated diagram then has at most 4
  diagrams. Using the result of \cite{Zib} then the corollary follows
  as before.
  \end{rem}

\section{Further properties}
\label{sec:further}
It is a standard fact that the knot Floer homology of the connected
sum of two knots is the tensor product of the knot Floer homologies:
\[
\HFKa (K_1\# K_2)\cong \HFKa (K_1)\otimes \HFKa (K_2).
\]
From this (bigraded) isomorphism it follows that
\[
th(K_1\# K_2)=th(K_1)+th(K_2).
\]
The behaviour of $\beta (K)$ is less clear under connected summing.
Suppose that $K_1, K_2$ are both non-alternating knots. By taking the
connected sum of two diagrams $D_1,D_2$ for these knots at bad edges
(i.e. arcs on the boundary of bad domains verifying that the domains
are bad) we get that
\[
B(D_1\# D_2)=B(D_1)+B(D_2)-2,
\]
immediately implying that
\[
\beta (K_1\# K_2)\leq \beta (K_1)+\beta (K_2)-2.
\]
Motivated by the equality for the thickness $th$, we conjecture
\begin{conj}
  If $K_1, K_2$ are two non-alternating knots, then
  \[
  \beta (K_1\# K_2)= \beta (K_1)+\beta (K_2)-2.
  \]
\end{conj}

\bigskip

It is not hard to find knot diagrams for which
Inequality~\eqref{eq:ThicknessBadness} is sharp. Indeed, the standard
diagram of the pretzel knot $P(-3,5,5)$ admits four bad domains (see
Figure~\ref{fig:equality}(a)), while an explicite calculation of
$\HFKa (P(-3,5,5))$ shows that $th(P(-3,5,5))=1$. Consider the
$n$-fold connected sum $K_n=\#_n P(-3,5,5)$;  connect summing the
diagrams at bad edges (in the above sense) we get a sequence of knots
$K_n$ and diagrams $D_n$ for them with the properties that $th
(K_n)=n$ and $B(D_n)=2n+2$, cf. Figure~\ref{fig:equality}(b).
The non-alternating knots $K_n$ then satisfy
$n=th(K_n)=\frac{1}{2}\beta (K_n)-1$.
\begin{figure}
\centering
\includegraphics[width=5cm]{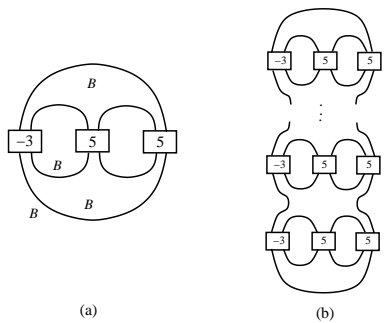}
\caption{ In (a) he pretzel knot $P(-3,,5,5)$ is shown. The $B$ symbols
  signify the bad domains. (A box containing
  the integer $n$ denotes $\vert n \vert $ half twists, right handed for $n>0$
  and left handed for $n<0$.) In (b) we provide a diagram of $K_n$, where the
connected sum is taken at bad domains.}
\label{fig:equality} 
\end{figure}

\bibliography{biblio} \bibliographystyle{plain}

\begin{thebibliography}{1}

\bibitem{ChampKof}
A.~Champanerkar and I.~Kofman.
\newblock A survey on the {T}uraev genus of knots.
\newblock {\em Acta Math. Vietnam.}, 39(4):497--514, 2014.

\bibitem{FiveAuthors}
O.~Dasbach, D.~Futer, E.~Kalfagianni, X.-S. Lin, and N.~Stoltzfus.
\newblock The {J}ones polynomial and graphs on surfaces.
\newblock {\em J. Combin. Theory Ser. B}, 98(2):384--399, 2008.

\bibitem{Goldman-Kauffman}
J.~Goldman and L.~Kauffman.
\newblock Rational tangles.
\newblock {\em Adv. in Appl. Math.}, 18(3):300--332, 1997.

\bibitem{Josh}
J.~Greene.
\newblock Alternating links and definite surfaces.
\newblock {\em Duke Math. J.}, 166(11):2133--2151, 2017.
\newblock With an appendix by Andr\'{a}s Juh\'{a}sz and Marc Lackenby.

\bibitem{Howie}
J.~Howie.
\newblock A characterisation of alternating knot exteriors.
\newblock {\em Geom. Topol.}, 21(4):2353--2371, 2017.

\bibitem{Lowrance}
A.~Lowrance.
\newblock On knot {F}loer width and {T}uraev genus.
\newblock {\em Algebr. Geom. Topol.}, 8(2):1141--1162, 2008.

\bibitem{OSzAlternating}
P.~Ozsv\'{a}th and Z.~Szab\'{o}.
\newblock Heegaard {F}loer homology and alternating knots.
\newblock {\em Geom. Topol.}, 7:225--254, 2003.

\bibitem{Turaev}
V.~Turaev.
\newblock A simple proof of the {M}urasugi and {K}auffman theorems on
  alternating links.
\newblock {\em Enseign. Math. (2)}, 33(3-4):203--225, 1987.

\bibitem{Zib}
C.~Zibrowius.
\newblock On symmetries of peculiar modules; or, $\delta$-graded link {F}loer
  homology is mutation invariant.
\newblock arXiv:1909.04267, 2019.

\end{thebibliography}

\end{document}